\documentclass{amsart}

\usepackage{mathptmx}      
\usepackage{graphics}
\usepackage{amstext,amsfonts,amsmath}
\usepackage{amssymb,amscd}
\usepackage{latexsym}
\usepackage{graphicx}

\newtheorem{theorem*}{Theorem}
\newtheorem{theorem}{Theorem}

\newtheorem{problem*}[theorem]{Problem}

\newtheorem{corollary}[theorem]{Corollary}
\newtheorem{proposition}[theorem]{Proposition}
\newtheorem{claim}{Claim}

\def\P{\mathbb{P}}

\def\N{\mathbb{N}}
\def\Z{\mathbb{Z}}
\def\S{\mathbb{S}}
\def\C{\mathbb{C}}
\def\CC{\mathcal{C}}

\usepackage{amstext,amsfonts,amsmath}
\usepackage{amssymb,amscd}

\usepackage{mathptmx}      
\usepackage{amstext,amsfonts,amsmath}
\usepackage{amssymb,amscd}
\usepackage{latexsym}
\usepackage{graphicx}

\newtheorem{assertion}{\bf Assertion} 

\def\P{\mathbb{P}}

\def\N{\mathbb{N}}
\def\Z{\mathbb{Z}}
\def\S{\mathbb{S}}
\def\C{\mathbb{C}}
\def\CC{\mathcal{C}}

\def\QQ{\overline{\mathbb{Q}}}

\begin{document}
\title{Jacobian conjecture as a problem on integral points on affine curves}
\author{Nguyen Van Chau}
\address{Institute of Mathematics, Vietnam Academy of Science and Technology, 18 Hoang Quoc Viet, 10307 Hanoi, Vietnam.}
\email{nvchau@math.ac.vn}
\thanks{The author was partially supported by Vietnam National Foundation for Science and Technology Development (NAFOSTED) grant 101.04-2017.12, and Vietnam Institute for Advanced Study in Mathematics (VIASM)}

\subjclass[2010]{14R15, 14R25, 11D72 }
\date{}
\keywords{Integral point, Polynomial automorphism, Jacobian conjecture}
\maketitle




\begin{abstract} 
It is shown that the Jacobian conjecture over  number fields may be considered as an existence problem of integral points on affine curves. More specially, if the Jacobian conjecture over $\C$ is false, then for some $n\gg 1$ there exists a counterexample $F\in \Z[X]^n$  of the  form 
$F_i(X)=X_i+ (a_{i1}X_1+\dots+a_{in}X_n)^{d_i}$, $a_{ij}\in \Z$, $d_i=2;3 $, $i,j=\overline{1,n},$
such that the affine curve $F_1(X)=F_2(X)=\dots=F_n(X)$ has no non-zero integer points.


\end{abstract}

\section{Introduction}
Let $k$ be a field of characteristic zero and $k[X]$ the ring of polynomials of the variable $X:=(X_1,X_2,\dots, X_n)$, $n>1$. In tradition,  polynomial maps $F=(F_1,F_2,\dots,F_n)\in k[X]^n$ with $JF:=\det DF\equiv 1$ will be called  {\it Keller maps}.  The  $n$-dimensional 
Jacobian conjecture over $k$ ($JC(k,n)$)
, which was posed firstly in 1939 by Ott-Heinrich Keller \cite{Keller}
and still open even for $n=2$, asserts that  {\it  every Keller map $F\in k[X]^n$ is invertible, and hence, has an inverse in $k[X]^n$}. We refer the readers to \cite{Bass,EssenBook} for nice surveys on this conjecture and related topics. 

It is known that the Jacobian conjecture is true if it is true for the field $\QQ$ of all algebraic numbers and that $JC(K,n)$ for a number field $K$ can be reduced to consider the existence of solutions in $O_K^n$  of the Diophantine equation $F(X)=0$ for Keller maps $F\in O_K^n[X]^n$, where $O_K$ is the ring of integers of $K$ (see in \cite{Bass,McKenna}). These facts show another aspect of the Jacobian conjecture from view points of the Diophantine geometry and would be useful in attempting to understand the nature of this conjecture.
This article is to present the following results that reduce the Jacobian conjecture  to an existence problem of  integral points on affine curves. 
 
\medskip
\noindent{\bf Theorem A}. {\it Let $K$ be a number field and $O_K$ the ring of integers in $K$. For every $n>1$ the conjecture $JC(K,n)$ is equivalent to the following statement:}
 \begin{enumerate}
\item[] $DJC(K,n)$: {\it  For every Keller map $F\in O_K[X]^n$ with $F(0)=0$  the affine curve 
$$F_1(X)=F_2(X)=\dots=F_n(X).\eqno(\CC_F)$$
always has non-zero points in $O_K^n$. } 
\end{enumerate}

\medskip

\medskip
\noindent{\bf Theorem B}. {\it If the Jacobian conjecture over $\C$ is not true,
 then  for some  $n\gg 1$ there exists a Keller map $F\in \Z[X]^n$ of 
the so-called quadratic-cubic linear form over $\Z$,
 $$F_i(X)=X_i+ \langle a_i,X\rangle^{d_i},\; a_i\in \Z^n,\;d_i=2;3 ,\;i=\overline{1,n}$$
, such that the affine curve $\CC_F$ has no non-zero points in $\Z^n$. }

\medskip

Theorem B leads to a little surprise consequence that the Jacobian conjecture over $\C$ for all $n>1$ can be reduced to the question whether for every Keller map $F\in \Z[X]^n$ of 
the quadratic-cubic linear form over $\Z$ 
the $6$-degree diophantine equation $F_1(X)^2+\dots+F_{n-1}(X)^2=0$ always has non-zero integer solutions (Cor. 1, Sec. 4).

Our approach here is based on  the celebrated Siegel's theorem on the integral points on affine curves and the  reduction theorems for the Jacobian conjecture, due to Bass, Connell and Wright  \cite{Bass}, Yagzhev  \cite{Yagzhev}, Dru\.zkowski  \cite{Druz1} and Connell and Van den Dries  \cite{Connell}. Theorem A and Theorem B will be proved in  sections 3 and 4. The essential key in the proofs is Main Lemma in Section 2, which shows that if $K$ is a number field and if  $F\in K[X]^n$ is a non-invertible Keller map, then for  general lines $l$ in $K^n$ the inverse images  $F^{-1}(l)$ are irreducible affine curves having at most finitely many points in $O_K^n$.

\section{Lemma on inverse images of generic lines} 
In this section we consider the possible behavior of non-invertible Keller maps, if exist. We will try to estimate the topological type and the number of integral points of the inverse images of generic lines  by such Keller maps. 

Let us denote $l(u,v):=\{ u+tv: t\in \C\}$ - the complex line of direction $0\neq v\in \C^n$ passing through a point $u\in \C^n$. By the {\it bifurcation value set}, denoted by $E_f$, of a dominant polynomial mapping $f: \C^n\longrightarrow \C^m$ we mean the smallest algebraic subset of $\C^m$ such that the restriction $f: \C^n\setminus f^{-1}(E_f)\longrightarrow \C^m\setminus E_f$ defines a locally trivial smooth fibration.

Suppose  $F\in \C[X]^n$ is a Keller map, but not invertible.  In this situation,  the field extension $\C(F)\subset \C(X)$ is algebraic and has  degree $d_F:=[\C(X):\C(F)] >1$. The bifurcation value set $E_F$ is just the set of all $a\in \C^n$ such that $\#F^{-1}(a)<d_F$, and the restriction $F:\C^n\setminus F^{-1}(E_F)\longrightarrow \C^n\setminus E_F$ gives a unbranched smooth covering of $d_F$ sheets. Since the extension $\C(F)\subset \C(X)$ is algebraic, there exists irreducible polynomials $h_i\in \C[X][T]$, $i=\overline{1,n}$, such that $h_i(F(X), X_i)= 0$. Following \cite{Jelonek01,EssenBook}, the bifurcation value set $E_F$ then is the hypersurface defined by the equation $a_1(X)\dots a_n(X)=0$, where $a_i(X)$ are  coefficients of  $T^{\deg_Th_i}$  in $h_i$. Let $H_F(X)$ be the product of all distinct irreducible polynomial factors of $a_i(X)$s. Then $E_F=\{a\in \C^n: H_F(a)=0\}$ and, in particular,  $H_F\in \QQ[X]$ if $F\in \QQ[X]^n$. Let us denote by  $K_F$  the cone of tangents at infinity of $E_F$, i.e. $K_F=\{v\in\C^n:  h_F(v)=0\}$, where $h_F$ is the leading homogeneous of $H_F$. The following lemma is an essential key in the proofs of Theorem A and Theorem B.

\medskip
\noindent{\bf Main Lemma } {\it
Suppose $F\in \C[X]^n$ is a Keller map, but not invertible. Then, there exists a non-constant polynomial $\sigma_F\in \C[U,V]$, $U=(U_1\dots,U_n)$ and $V=(V_1,\dots,V_n)$, such that 
\begin{enumerate}
\item[a)]  $\sigma_F(\bar u,V)\not\equiv 0$ and $\sigma_F(U,\bar v)\not\equiv 0$ for $\bar u\not \in E_F$ and $\bar v\not\in K_F$, and
\item[b)] the inverse images $F^{-1}(l(u,v))$ with $ u,v\in \C^n$ and $\sigma_F(u,v)\neq 0$ 
are  irreducible affine curves of
the same genus $g_F$ and  number $n_F>2$ of irreducible branches at infinity.
\end{enumerate}
\noindent In the other words, the inverse image by $F$ of a generic line is a connected Riemann surface with more than two punctures.
Moreover, in the case $F\in K[X]^n$ for a number field $K$, 
\begin{enumerate}
\item[c)] $H_F,\sigma_F\in \QQ[X]$ and 
\item[d)] for every $ u,v\in K^n$ with $ \sigma_F(u,v)\neq 0$ the inverse image $F^{-1}(l(u,v))$  has at most finitely many integral points in $K^n$.
\end{enumerate}
}

\medskip
\begin{proof} Let $F:\C^n\longrightarrow \C^n$ be as in the statement. We will try to construct the desired polynomial $\sigma_F$ and  prove the conclusions (a-c). Conclusion (d) then is an immediate consequence of (a), (b) and  Siegel's theorem, which asserts that any irreducible affine curve over a number field $K$ with positive genus or with more than two irreducible branches at infinity has at most finitely many $K$-integral points.

In view of the results on the topological equisingularity, due to Verdier \cite{Verdier} and Varchenko \cite{Varchenko}, there exists a proper algebraic subset $\Sigma\subset\C^n\times\C^n$ such that  the inverse images $F^{-1}(l)$, $l:=l(u,v)$, $(u,v)\in (\C^n\times\C^n)\setminus\Sigma$, are diffeomorphic to 
the same connected Riemann surface. The polynomial $\sigma_F$ then may be taken to be a  polynomial that defines an algebraic hypersurface containing $\Sigma$. Below, we will determine $\sigma_F$ by a constructive way that enables us to handle other properties in (a-c). To do that we need to use the following fact due to Jelonek \cite{Jelonek99,Jelonek01}: 

\medskip\noindent 
{\bf Theorem }(*)\label{JelonekThm} 
 {\it  Let $f:V\longrightarrow W$ be a dominant generically-finite polynomial mapping of irreducible affine varieties $V\subset \C^n$ and $W\subset \C^m$. Let $S_f$ be the so-called non-proper value set  of $f$, $S_f:=\{ a\in W: a=\lim_kf(x_k),  V\ni x_k\mapsto \infty\}$. Then, $S_f$ is empty or an algebraic hypersurface in $W$. Moreover, if the coefficients in the polynomials defining $V,W$ and $f$ are algebraic numbers, then all coefficients in the polynomial defining $S_f$ are also algebraic numbers.}
{\rm (see Theorem 2.3, \cite{Jelonek01} and it's proof.)}

\medskip
Now, consider the $d_F$-sheeted unbranched covering 
$$F:\C^n\setminus F^{-1}(E_F)\longrightarrow \C^n\setminus E_F.\eqno(1)$$
Since $F$ is locally diffeomorphic, the bifurcation value set $E_F$ coincides with the non-proper value set $S_F$ of $F$. By Theorem (*) we have
\begin{enumerate}
\item[i)] $E_F$ is an affine hypersurface in $\C^n$ defined by a reduced polynomial $H_F\in \C[X]$. If $F\in \QQ[X]^n$, then $H_F\in\QQ[X]$.
\end{enumerate}
Let $W\subset E_F$  be an irreducible component of $E_F$,  given by $h_W(X)=0$ for an irreducible factor $h_W\in \C[X]$ of $H_F$. Let $V$ be  an irreducible component of the inverse image $F^{-1}(W)$ and $f_{VW}:=F_{V}:V\longrightarrow W.$  Denote by $\Sigma(W)$ the set of all singular points of $W$ and by $E_{VW}$ the non-proper value set of $f_{VW}$.  Again by Theorem (*) 
\begin{enumerate}
\item[ii)] $E_{VW}$ is either empty set or  an algebraic variety of dimension $n-2$,  given by the equations  $h_W(X)=0$ and $g_{VW}(X)=0$ for a $g_{VW}\in \C[X]$;
\item[iii)] the restriction $f_{VW}: V\setminus f_{VW}^{-1}(\Sigma(W)\cup E_{VW})\longrightarrow W\setminus (\Sigma(W)\cup E_{VW})$ gives a unramified covering;
\item[iv)] If  $F\in\QQ[X]^n$, then $h_W,g_{VW}\in\QQ[X]$.
\end{enumerate}
Let  $B_F:=\bigcup E_{VW}$, where $W$ and $V$  
 run over the irreducible components of $E_F$ and the irreducible components of $F^{-1}(W)$, respectively.
Let $E^1_F$ denote the bifurcation value set of the restriction $F:F^{-1}(E_F)\longrightarrow E_F$. By definitions  $E^1_F$ is a proper algebraic subset of $E_F$, $E_F\setminus E^1_F$ is smooth and  the restriction  $$F:F^{-1}(E_F\setminus E^1_F)\longrightarrow E_F\setminus E^1_F\eqno (2)$$
gives a unramified  covering.  By (ii-iv) we can see  
\begin{enumerate}
\item[v)] $E^1_F$ is the union of $B_F$ and the set of all singular points of $E_F$.
\end{enumerate}
Let us denote by $\Sigma_F$ the complement of the set of all $(u,v)\in \C^n\times \C^n$ satisfying the following conditions: 
\begin{enumerate}
\item[a1)] $l(u,v)$  intersects transversally $E_F$,
\item[a2)] $v\not \in K_F$ and
\item[a3)] $l(u,v)\cap E^1_F=\emptyset$. 
\end{enumerate}
By (v) we can easy verify that  $(u,v)$ satisfies the conditions (a1-a3) if and only if it satisfies 
\begin{enumerate}
\item[b1)] $l(u,v)$ intersects $E_F$ at $\deg H$ different points, and
\item[b2)] $l(u,v)\cap B_F=\emptyset$.
\end{enumerate}
Now, we define 
$$\begin{aligned}&D(U,V)&:=&Disc_t(H(U+tV))\cr 
&R(U,V)&:=&\prod_W\prod_{VW} Res_t(h_W(U+tV),g_{VW}(U+tV))\cr
&\sigma_F(U,V)&:=&D(U,V)R(U,V).
\end{aligned}
$$
\begin{assertion} We have
\begin{enumerate}
\item[c1)] $\Sigma_F=\{(u,v)\in \C^n\times\C^n:\sigma_F(u,v)= 0\};$
\item[c2)] $\sigma_F(\bar u,V)\not\equiv 0$ and $\sigma_F(U,\bar v)\not\equiv 0$ for $\bar u\in \C^n\setminus E_F$ and $0\neq \bar v\in \C^n\setminus K_F$;
\item[c3)]  $\sigma_F\in \QQ[U,V]$ if $F\in \QQ[X]^n$.
\end{enumerate} 

\end{assertion}
\begin{proof} Observe that the conditions (b1) and (b2) can be expressed as $D(u,v)\neq 0$ and $R(u,v)\neq 0$, respectively, that implies (c1). (c3) follows from the constructing of $\sigma_F$ and (iv). 

We now prove (c2).

Given $\bar u\in \C^n\setminus E_F$. For each irreducible component $W$ of $E_F$ let $W_{\bar u}:=\{ x\in W: \varphi(x)=0\}$ and $V_{\bar u}:=\bigcup_WW_{\bar u}$, where $\varphi(X):= Dh_W(X)(X-\bar u)$. We define the cone $$K_{\bar u}:=K_F\cup \{tv: v\in V_{\bar u}\cup E^1_F, t\in \C\}.$$ Then, by the conditions (a1-a3) $\sigma_F(\bar u,v)\neq 0$  if and only if $v\in \C^n\setminus K_{\bar u}$. Obviously, $\deg \varphi(X)=\deg h_W(X)$ and $\varphi(\bar u)=0$. Since $h_W$ is irreducible and $\bar u\not\in E_F$, it follows that  $W_{\bar u}$ is a subset of $W$ and of pure dimension less than $n-1$. Therefore, $V_{\bar u}$ is an algebraic subset of $W$ and has a pure dimension less than $n-1$. Note that  $\dim E^1_F <n-1$ and $\dim K_F<n$. This follows that  $K_{\bar u}$  is a closed set of pure dimension less than $n$. Thus, $\C^n\setminus K_{\bar u}$ is open dense in $\C^n$, and hence, $\sigma_F(\bar u, V)\not\equiv 0$.

Given $\bar v\not \in K_F$. For each irreducible component $W$ of $E_F$ let $W_{\bar v}:=\{ x\in W: D h_W(x)\bar v=0\}$ and $U_{\bar v}:=\bigcup_WW_{\bar v}$. We define $S_{\bar v}:=\{u\in \C^n:u=x+t\bar v, t\in \C, x\in U_{\bar v}\cup E^1_F\}$. Since $\bar v\not \in K_F$, by the conditions (a1-a3) we can verify that $\sigma_F(u,\bar v)\neq 0$ if and only if $u\not\in S_{\bar v}$. Furthermore, $D h_W(X)\bar v\not\equiv 0$. Otherwise, $\bar v$ belongs to the tangent space $T_xW$ for all smooth point $x$ in $W$ that is impossible. Obviously,  $\deg Dh_W(X)\bar v \leq \deg h_W(X)-1$.  So, $W_{\bar v}$ is a proper algebraic subset of $W$ and has a pure dimension less than $n-1$. Hence,  $U_{\bar v}$ is of pure dimension less than $n-1$. Again we can  see that $S_{\bar v}$ is of dimension less than $n$ that ensures $\sigma_F(U,\bar v)\not \equiv 0.$

\end{proof}

\begin{assertion} The family $\mathcal{F}:=\{F^{-1}(l(u,v)): \sigma_F(u,v)\neq 0\}$ consists of  irreducible affine  curves of 
the same topological type.
\end{assertion}
\begin{proof} It suffices to show 
\begin{enumerate}
\item[d1)] $\mathcal{F}$ contains an irreducible affine curve, and
\item[d2)] The affine curves in $\mathcal{F}$ are homeomorphic.
\end{enumerate}

We first prove (d1). Fix  $0\neq  v\in \C^n\setminus K_F$. 
Let $E$ denote the vector space orthogonal to $v$, $E=\{ u\in \C^n: \langle u,v\rangle=0\}$. We will show that there is an open dense algebraic subset $E'$ of $E$ such that for every $z\in E'$ the curve $F^{-1}(l(u,v))$ is irreducible and $\sigma_F(u,v)\neq 0$. Let $\pi: \C^n \longrightarrow E\cong \C^{n-1}$ the projection $u\mapsto u-\frac{\langle u,v\rangle}{\langle v,v\rangle}v\in E$ and $\varphi:=\pi\circ F:\C^n\longrightarrow E$. We have that  $F^{-1}(l(u,v))=\varphi^{-1}(u)$ for $u\in E$. Observe that $\varphi$ is a dominant morphism. Moreover, the restriction $\pi: E_F\longrightarrow E$ is proper, since $v\not\in K_F$. Then, by Theorem 2 in \cite{Krasinski} there exists an open dense algebraic subset $A$ of $E$ such that for every $u\in A$ the fiber $\varphi^{-1}(u)$ is irreducible. On the other hand, since  $ v\not \in K_F$, by Assertion 1 (c2) the set $B:= \{u\in E: \sigma_F(u,\bar v)\neq 0\}$ is open dense in $E$. So, $E':=A\cap B$  provides the desired subset. 

To prove (d2) it suffices to show that for two arbitrary $z^i\in \C^n\times \C^n$ with $\sigma_F(z^i)\neq 0$, $i=0;1$, the curves $F^{-1}(l(z^i))$  are homeomorphic. Given such points $z^i$. Since the polynomial $\sigma_F$ is not constant,  we can take  a smooth path $\gamma:[0,1]\longrightarrow \C^n\times \C^n$ such that
$\gamma(0)=z^0$,  $\gamma(1)=z^1$ and $\sigma_F(\gamma(t))\neq 0$ for $t\in [0,1]$.
Let $l_t:=l(\gamma(t))$, $t\in [0,1]$. Note that $E_F\setminus E^1_F$ is a smooth open algebraic subset of the hypersurface $E_F$ and by(1-2) the restriction
 $$F:(\C^n\setminus F^{-1}(E_F), F^{-1}(E_F\setminus E^1_F))\longrightarrow (\C^n\setminus E_F,E_F\setminus E^1_F)\eqno(3)$$
 defines a unramified covering. Furthermore, by the construction of $\sigma_F$ the lines $l_t$  intersect transversally $E_F$,  do not pass through $E^1_F $ as well as   tangent to  $E_F$  at infinity. Then 
 by a standard way we can construct a continue family of homeomorphisms 
$$\Gamma_t: ( l_0, l_0\setminus E_F,l_0\cap E_F)\longrightarrow (l_t, l_t\setminus E_F, l_t\cap E_F),\; t\in [0,1].$$
The lift of $\Gamma_t$ by the unramified covering (3) then introduces homeomorphisms
$$\Phi_t:(F^{-1}(l_0),F^{-1}(l_0\setminus E_F),F^{-1}(l_0\cap E_F))\longrightarrow (F^{-1}(l_t),F^{-1}(l_t\setminus E_F),F^{-1}(l_t\cap E_F)).$$
In particular, $\Phi_1$ gives a homeomorphism of $F^{-1}(l_0)$ and $F^{-1}(l_1)$.
\end{proof}

Let $(g_F,n_F)$ denote the topological type of curves in $\mathcal{F}$, where $g_F$  and $n_F$ are genus and the number of  irreducible branches at infinity, respectively. 
\begin{assertion}\label{LemnF} $n_F>2$.
\end{assertion}
\begin{proof} 
We first show that there exists  $(u,v)\in \C^n\times \C^n$ such that $\sigma_F(u,v)\neq 0$ and the line $l:=l(u,v)$ is contained in the image $F(\C^n)$. To see it, fix $u\in \C^n\setminus E_F$ and let $V:=\{v\in \C^n: v\neq 0, \sigma(u,v)\neq 0\}$. By Lemma 4  $V$ is open dense in $\C^n$ and $\sigma_F(u,v)\neq0$ for all $v\in V$.
 Let $S:=\C^n\setminus F(\C^n)$ and $K_u:=\{v\in \C^n: v=t (s-\bar u), s\in S, t\in \C\}$.  Then  $l(u,v)$ 
intersects 
$S$ if and only if $v\in K_u$. So it suffices to show $V\setminus K_{u}\neq \emptyset$. Observe,  $S$ is a closed proper algebraic subset of $\C^n$, since $F$ is locally diffeomorphic on $\C^n$. Furthermore,  $\dim S<n-1$. Otherwise, $S$ would contains a hypersurface $h(X)=0$ for a non-constant polynomial $h\in \C[X]$. This would imply that $h\circ F(X)\equiv c\neq 0$, and consequently, $Dh(F(X)).DF(X)\equiv 0$ that is impossible. So, the cone $K_v$ is of pure dimension $<n$. Hence  $V\setminus K_u$ is open dense in $\C^n$. 

Now, we will prove $n_F>2$. By the above observation we can take a line $l=l(u,v)$ such that $\sigma_F(u,v)\neq 0$  and $l\subset F(\C^n).$ Let $C:= F^{-1}(l)$ and  $\hat C$ be a smooth compactification of $C$. By Assertion 2 $\hat C$ is a connected Riemann surface of genus $g_F$ and $\hat C\setminus C$ consists of $n_F$ distinct points.  
Regarding $l$ as the line $\C$ in $\P^1=\C\cup\{\infty\}$, we can extend the restriction of $F$ on $C$ to a regular morphism $f:\hat C\longrightarrow \P^1$ that gives a ramified covering of $d_F$-sheets over $\P^1$. The morphism   $f$ may be ramified only at points in $\hat C\setminus C$, since $JF\equiv 1$. Moreover, by the choice of $l$ we have $\P^1\setminus f(C)=\{\infty\}$.  Now, applying  Hurwitz Relation to $f$ we have
$$2-2g_F=2d_F- \sum_{a\in \hat C\setminus C}(\deg_{a}f-1)\eqno(4)$$
, where $\deg_a f$ is the local degree of $f$ at $a$. Note that $\deg_af\leq d_F$ and $\deg_af= d_F$ only if $f^{-1}(\{f(a)\})=\{a\}$. From (4) it follows  that  $g_F=0$ and $d_F=1$ for when $n_F=1$ and that $g_F=0$ and $\P^1\setminus f(C)$ consists of two distinct points for when $n_F=2$. So, both of the cases $n_F=1;2$ are impossible.
\end{proof}

\medskip
All (a-c)   now follow from the assertions 1-3.
\end{proof}
\section{Proof of Theorem A}

In sequels, for any number field $K$ we denote by $\S^n_K$ the set of all primitive vectors in $O_K^n$, $\S^n_K:=\{v=(v_1,\dots,v_n)\in O_K^n: \gcd\{v_i\}=1\}$. We will prove the following variant of Theorem A.

\medskip
\noindent{\bf Theorem A'} {\it  Let $K$ be a number field $K$. The conjecture $JC(K,n)$, $n>1$, is equivalent to the statement:
\begin{enumerate}
\item[]$DJC(K,n,m)$: Let $m\in \N$,  $n> m\geq 0$. For every Keller map $F\in O_K[X]$ with $F(0)=0$ the affine curve 
$$\begin{cases}F_1(X)=\dots=F_m(X)=0 \\
F_{m+1}(X)=\dots=F_n(X)\end{cases}\eqno \mathcal{C}_F(m)$$ has  non-zero points in $O_K^n$.
\end{enumerate}}

\begin{proof} We need to use the following elementary fact.
\begin{claim} {\it Let $K$ be a number field $K$. For  $v,w\in \S^n_K$ there is a matrix $A\in SL(n,O_K)$ such that $Av=w$.}
\end{claim}
 \begin{proof} It suffices to show  that for  each $v\in \S_K^n$ there is  $A\in SL(n,O_K)$ such that $ A e_1=v$, where $e_1:=(1,0,0,...,0)^T$. We do prove that by induction on $n$. The cases $n=1,2$ are obvious.  Let 
$n\geq 3$ and let $v=(v_1,v_2,\dots,v_n)^T\in \mathrm{S}^n_K$ be given. Without  a loss of generality we may assume 
that  $v_1\neq 0$. Let $\bar v:=r^{-1}( v_2,..., v_n)^T$ with $r:=\gcd(v_2,...,v_n)$. Note that $\gcd(r,v_1)=1$.  
By induction assumption there is a $(n-1)\times(n-1)$ matrix $\bar A:=[\bar v\;\; B]\in SL(n-1,O_K)$, for which $\bar A\bar e_1=\bar v$
, where $\bar e_1=(1,0,\dots,0)^T\in \S_K^{n-1}$. Let 
$$A(\alpha,\beta):= \begin{bmatrix}v_1&O&\beta\\r\bar v&B&\alpha\bar v
 \end{bmatrix},$$
which is a $n\times n$-matrix with entries in $O_K$ and parameters  $ \alpha,\beta\in O_K$.
Obviously, $A(\alpha,\beta)e_1=v$ and $\det A(\alpha,\beta)=\alpha v_1+(-1)^n \beta r$. Since $\gcd(r,v_1)=1$, we can choose  $0\neq \alpha_0,\beta_0\in O_K$ such that $\det A(\alpha_0,\beta_0)=1$, i.e. $A(\alpha_0,\beta_0)\in SL(n,O_K)$. This concludes the proof.
 \end{proof}

To prove the theorem it suffices to verify that if $JC(K,n)$ is false, then $DJC(K,n,m)$ has a counterexample. Assume that $G\in K[X]^n$ is a non-invertible Keller map $G\in K[X]^n$. We will show that there is a non-invertible Keller map $H\in O_K[X]^n$ with $H(0)=0$ and a line $l=l(0,v)$, $v\in \S^n_K$, such that $H^{-1}(l)\cap O_K^n=\{0\}$. Then, by the above claim we may choose a matrix $A\in SL(n,O_K)$ such that $Av=w$, where $w_i=0$ for $i=\overline{1,m}$ and $w_i=1$ for $i=\overline{m+1,n}$. The map $F(X):=A\circ H\circ A^{-1}(X)$ then provides a counterexample to $DJC(K,n,m)$.

Now, we determine such desired $H$ and $v$. By changing coordinates, if necessary, we may assume that $G\in O_K[X]^n$, $G(0)=0$ and $0$ is not a bifurcation value of $G$. Note that the algebraic closure of $\S_K^n$ is the whole space $\C^n$. Since $0\not \in E_G$,  by Main Lemma  we may take a line $l:=l(0,v)$ with $v\in \S^n_K$ and $\sigma_F(0,v)\neq 0$ such that $S:=G^{-1}(l)\cap O_K^n$ is a finite set. If $S=\{0\}$, we put $H(X):=G(X)$. If $S\neq \{0\}$, we may choose a number $0\neq r\in O_K$ such that $S\cap rO_K^n=\{0\}$, and then put 
$H(X):=\frac{1}{r}G(rX)=G^{(1)}(X)+rG^{(2)}(X)+\dots+r^{k-1}G^{(k)}(X),$
where $G^{(i)}$ are $i$-degree homogeneous components of $G$. Observe, $H\in O_K[X]^n$  and is a Keller map. Moreover,  if $x\in H^{-1}(l)\cap O_K^n$, then $G(rx)=r H(x)\in l$, i.e. $rx\in S\cap O_K^n$. Therefore,  $x=0$ by the choice of $r$. Hence, $H^{-1}(l)\cap O_K^n=\{0\}$.
\end{proof}

\medskip

\section{Proof of Theorem B}

First, we recall the following elementary properties of the bifurcation value sets, which will be used later.  
\begin{proposition}\label{Pro2} Let $F\in \C[X]^n$  with $F(0)=0$. 
\begin{enumerate}
\item[i)]  $E_{F\circ P}=E_F$ and $E_{P\circ F}=P(E_F)$ for  polynomial automorphisms $P$ of $\C^n$.
\item[ii)] $E_{(^rF)}=\frac{1}{r} E_F$ for  $(^rF)(X):=\frac{1}{r}F(rX)$, $0\neq r\in \C$;
\item[iii)] $E_{\hat F} =E_F\times \C^m$ for $\hat F(X,Y):=(F(X),Y)\in \C[X,Y]^{n+m}$.
\end{enumerate}
\end{proposition}
The proof is elementary and left to the readers.



\medskip

\noindent{\bf Theorem C} {\it If the Jacobian conjecture over $\C$ is false, then for some $n\gg 1$ it has a counterexample $F\in \Z[X]^n$ of the  quadratic-cubic linear form over $\Z$
such that   $0\not \in E_F$.}

\medskip

\begin{proof}  
Assume the Jacobian conjecture over $\C$ is false. It follows from 
 Theorem 1.5 in \cite{Connell} that
for some $m\gg 1 $ there exists a non-invertible Keller map of $\C^m$ with coefficients in $\Z$. Then, by the reduction procedures due to Bass, Connell and Wright \cite{Bass}, Yagzhev  \cite{Yagzhev} and Dru\.zkowski  \cite{Druz1} for some $n>m$ we can construct a non-invertible Keller map $H\in \Z[X]^n$ of the cubic linear form $H_i(X)=X_i+ \langle c_i,X\rangle^3,\; c_i\in \Z^n,\; i=\overline{1,n}.$ 

Note that by the Jacobian condition the algebraic closure of $H(\Z^n)$ is the whole $\C^n$.  So, we can always take a point $\alpha\in \Z^n$ such that $H(\alpha)\not \in E_F$. 
Introduce a new variable $Y:=(Y_1,\dots, Y_n)$ and put $G(X,Y):=(H(X+\alpha)-H(\alpha), Y).$ Then, $G\in \Z[X,Y]^{2n}$ and  
$$\begin{aligned}
&G_i(X,Y)&=&X_i+3\langle c_i,\alpha\rangle^2\langle c_i,X\rangle+3\langle c_i,\alpha\rangle\langle c_i,X\rangle^2+ \langle c_i,X\rangle^3\cr
&G_{n+i}(X,Y)&=&Y_i\cr
\end{aligned}
$$
for all $i=\overline{1,n}$ . Furthermore, $G$ is a non-invertible Keller map, $E_G=(E_H-H(\alpha))\times\C^n$  by Proposition 1 and $0\not\in E_G$ by the choice of $\alpha$.

Now we define polynomial automorphisms of $\Z^{2n}$
$$
\begin{aligned}
&P(X,Y):=(X_1-3\langle c_1,\alpha\rangle Y_1,\dots,X_n-3\langle c_n,\alpha\rangle Y_n,Y_1,\dots,Y_n)\cr
&Q(X,Y):=(X_1,\dots,X_n, Y_1+\langle c_1,X\rangle^2,\dots, Y_n+\langle c_n,X\rangle^2)\cr
\end{aligned}
$$
and put   
$\hat F(X,Y):=P\circ G\circ Q(X,Y).$
Then, $\hat F\in \Z[Y,Y]^{2n}$ is a non-invertible Keller map and $0\not\in E_{\hat F}$ (by Proposition 1). Moreover,  
$$\begin{aligned}&\hat F_i(X,Y)&=&X_i-3\langle c_i,\alpha\rangle Y_i+ 3\langle c_i,\alpha\rangle^2\langle c_i,X\rangle +\langle c_i,X\rangle^3 \cr
&\hat F_{n+i}(X,Y)&=&Y_i+\langle c_i,X\rangle^2
\end{aligned}
$$
for all $i=\overline{1,n}$ .

Finally,  put $F(Z):=\hat F\circ R(Z)$, where $Z:=(X,Y)$ and $R$ denotes the inverse of the derivative matrix $D\hat F(0)$. Then, $F\in \Z[Z]^{2n}$ is a non-invertible Keller map with $0\not\in E_F$ and  
$$\begin{aligned}
&F_i(Z)&=& Z_i+\langle a_i, Z\rangle^3\cr
&F_{n+i}(Z)&=&Z_{n+i}+\langle a_i, Z\rangle^2
\end{aligned}
$$
for all $i=\overline{1,n}$, where $a_i:= R^T(c_i,0)\in \Z^{2n}$ and $R^T$ is the transpose of the matrix $R$.


\end{proof}

\medskip
\noindent{\bf Proof of Theorem B} Assume that the Jacobian conjecture is false.
Then, by Theorem C there exists  a non-invertible Keller map $F\in \Z[X]^n$ of 
the quadratic-cubic linear form over $\Z$ $$F_i(X)=X_i+\langle c_i,X\rangle^{d_i},\;  c_i\in \Z^n,\;d_i=2;3 ,\; i=\overline{1,n}$$ such that $0\not\in E_F$. 
Let $\sigma_F(U,V)\in \QQ[U,V]$ be as in Main Lemma. Since $0\not\in E_F$, by Main Lemma (a) we have $\sigma_F(0,V)\not\equiv 0$.  Note that the algebraic closure of $\{ (w_1^6,\dots, w_n^6): (w_1,\dots,w_n)\in \Z^n\}$ is the whole $\C^n$. So  we can choose  a vector  $w\in \Z^n$ with all $w_i\neq 0$   such that $\sigma_F(0,v)\neq 0$ for  $v:=(w_1^6,\dots,w_n^6)$. Then by Main Lemma (b) the curve $F^{-1}(l(0,v))$ has at most finitely many  points in $\Z^n$. Therefore, there is $0\neq r\in \Z$ such that  the curve $F^{-1}(l(0,v))$ has no non-zero  points in $r\Z^n$. Let $\hat F(X):=\frac{1}{r^3}F(r^3X)$. Then, 
$\hat F\in \Z[X]^n$ is a non-invertible Keller map, the curve $\hat F^{-1}(l(0,v))$ has no non-zero integer points and
$\hat F_i(X)=X_i+\langle b_i,X\rangle^{d_i},\; i=\overline{1,n}$, where $b_i:=r^{\frac{6}{d_i}}c_i\in \Z^n$.

Now, let $T:\C^n\longrightarrow \C^n$ be the linear transform of $\C^n$ defined by $T_i(X)=v_iX_i$ and
$$\delta:=\prod_{i=1}^nv_i=\prod_{i=1}^n w_i^6.$$ 
We define 
$$G(X):=\frac{1}{\delta}T^{-1}\circ \hat F\circ T(\delta X).$$ Then, $G$ is also a non-invertible Keller map and
$$G_i(X)= X_i+ \frac{1}{\delta}\frac{1}{v_i}\langle b_i, T(\delta X)\rangle^{d_i} =X_i+ \frac{\delta^{d_i-1}}{v_i}\langle T(b_i), X)\rangle^{d_i}=X_i+\langle a_i,X\rangle^{d_i}$$
for all $i=\overline{1,n}$, where $a_i:=(w_i^{-\frac{6}{d_i}}\prod_{j=1}^nw_j^{(d_i-1)\frac{6}{d_i}})T(b_i)\in \Z^n$ as $d_i=2;3 $.
Furthermore, the curve $\mathcal{C}_G$ has no non-zero integer point. Indeed, if $x\in \mathcal{C}_G\cap \Z^n$, $G(x)=t(1,\dots,1)$,  then $\frac{1}{\delta}\hat F\circ T(\delta x)=T\circ G(x)=tv$, and hence, $T(\delta x)\in \hat F^{-1}(l(0,v))\cap  \Z^n$. So, $x=0$  by the choice of $\hat F$.
Thus, $G$ provides a desired counterexample. 
\qed

\begin{corollary}\label{JC} To prove the Jacobian conjecture over $\C$  for all $n>1$ it suffices to verify that there does not exist Keller maps $F\in \Z[X]^n$ of the quadratic-cubic linear form over $\Z$ 
such that the equation $F_1(X)^2+\dots+F_{n-1}(X)^2=0$  has no non-zero solutions in $\Z^n$.
\end{corollary}

 \begin{proof} Assume the Jacobian conjecture over $\C$ is false. It is sufficient to show that for some $n\gg 1$ there is a Keller map $F\in \Z[X]^n$ of the quadratic-cubic linear form over $\Z$ such that the equation system $F_1=\dots=F_{n-1}=0$ has no non-zero points in $\Z^n$.

First, by the assumption and Theorem B for some $m\gg 1$ we may find a Keller map $G\in \Z[X]^m$ of the quadratic-cubic linear form over $\Z$, 
$$G_i(X)=X_i+\langle a_i,X\rangle^{d_i},\; a_i\in \Z^m,\; d_i=2;3,\;i=\overline{1,m}$$, such that the affine curve $\CC_G$  has no non-zero points in $\Z^n$. Introduce a new variable $T$ and define a new Keller map ,
$$F(X,T):=\hat G\circ R(X,T))\in \Z[X,T]^{m+1}$$, where $\hat G(X,T):= (G_1(X)-T,\dots, G_n(X)-T,T)$ and $R$ is the linear transformation $(X,T)\mapsto (X+Te,T)$, $e:=(1,\dots,1)\in \Z^m$. We have
$$\begin{aligned}&F_i(X,T)&=& G_i(X+Te)-T=X_i+\langle a_i,X-Te\rangle^{d_i},\;  i=\overline{1,m}\cr 
&F_{m+1}(X,T)&=&T\end{aligned}.\eqno(5)$$
In particular, $F$ is of the quadratic-cubic linear form over $\Z$.
 
Now, assume that $(x,t)\in \Z^{m+1}$ is a solution of the equation system $F_1=\dots=F_m=0$.
From (5) it follows that $G_1(x-te)=\dots=G_m(x-te)=t$ that means $x-te\in \CC_G$. Since $\CC_G\cap \Z^m=\{0\}$ and $G_i(0)=0$, we get $(x,t)=0$. Thus, the equation system $F_1=\dots=F_m=0$ has no non-zero solution in $\Z^{m+1}$. 
\end{proof}

\noindent{\it Acknowledgement}. The author wishes to express his thank to the professors Arno  van den Essen,  Ha Huy Vui and Ludwick. M. Dru\.zkowski  for some valuable discussions. The author would like to thank  the Vietnam Institute for Advanced Study in Mathematics for their helps.

\end{document}